\documentclass[reqno,12pt]{amsart}
\usepackage{hyperref}
\usepackage{amsmath,amsthm,amsfonts,amssymb}
\usepackage[cp1251]{inputenc}
   \usepackage[english]{babel}

\date{} 
\newtheorem{theorem}{Theorem}[section]
\newtheorem{lemma}[theorem]{Lemma}
\newtheorem{proposition}[theorem]{Proposition}
\newtheorem{remark}[theorem]{Remark}
\newtheorem{corollary}[theorem]{Corollary}
\newtheorem{example}[theorem]{Example}
\numberwithin{equation}{section}

\let\svthefootnote\thefootnote
\newcommand\freefootnote[1]{%
  \let\thefootnote\relax%
  \footnotetext{#1}%
  \let\thefootnote\svthefootnote%
}

\begin{document}

\begin{title}[Kantorovich type topologies]{Kantorovich type topologies on spaces of measures and convergence of barycenters}

\end{title}
\maketitle

\centerline{{\large Konstantin A. Afonin}}

\vskip .1in

\centerline{\footnotesize{ Department of Mechanics and Mathematics, Lomonosov Moscow State
University, Moscow, Russia;}}
\centerline{\footnotesize{ Moscow Center of Fundamental and Applied Mathematics, Moscow, Russia
}}

\vskip .2in

\centerline{{\large Vladimir I. Bogachev$^{*}$\freefootnote{$^{*}$corresponding author}}}

\vskip .1in


\centerline{\footnotesize{Department of Mechanics and Mathematics, Lomonosov Moscow State
University, Moscow, Russia;}}

\centerline{\footnotesize{
National Research University Higher School of Economics, Moscow, Russia
}}

\vskip .3in

{\footnotesize Abstract.
We study two topologies $\tau_{KR}$ and $\tau_K$ on the space of measures on a completely regular space
generated by Kantorovich--Rubinshtein and Kantorovich seminorms analogous to their classical norms in the case of a metric
space. The Kantorovich--Rubinshtein topology $\tau_{KR}$ coincides with the weak topology on nonnegative measures
and on bounded uniformly tight sets of measures.
A~sufficient condition is given for the compactness in the Kantorovich topology.
We show that for  logarithmically  concave  measures and stable measures weak convergence implies
convergence in the Kantorovich topology.
We also obtain an efficiently verified condition for convergence of the barycenters
of Radon measures from a  sequence or net weakly converging on a locally convex space.
As an application it is shown that for weakly convergent
 logarithmically  concave  measures and stable measures convergence of their barycenters
holds without additional conditions. The same is true for measures given by polynomial
densities of a fixed degree with respect to logarithmically  concave measures.
}

\freefootnote{2020 Mathematics Subject Classification: 28A33, 28C15, 46G12, 60B10}

\freefootnote{Keyword: Radon measure, Kantorovich distance, weak topology,
barycenter, logarithmically  concave  measure, Gaussian measure}

\freefootnote{This research is supported by the
Russian Science Foundation  Grant 22-11-00015}

\section{Introduction}

The geometry and topology of spaces of measures on metric spaces have become an important direction in probability theory
over the past two decades. A particular role in these studies is played by the Kantorovich metrics $W_p$ (also called
Wasserstein metrics in part of the literature) and similar Kantorovich--Rubinshtein (or Fortet--Mourier) metrics,
see, for example, \cite{AGS}, \cite{Back}, \cite{Back2}, \cite{Barbie}, \cite{BobLed},
\cite{B18}, \cite{Dolb}, \cite{Hellwig}, \cite{LieroMS},
\cite{Sturm06},
 \cite{VerZP}, \cite{Vill2}.
These metrics are traditionally considered on probability or nonnegative measures, where they are related to the weak topology,
but are also defined on all measures, although on the whole space of signed measures they induce topologies not comparable
with the weak one. However, for more general spaces analogous constructions have not been studied in detail,
although they were already considered by Castaing, Raynaud de Fitte and Valadier \cite[Section~3.4]{Cast}
in the framework of Young measures.
The goal of our paper is to further develop Kantorovich type seminorms in the case of completely regular spaces.
We consider two natural locally convex topologies $\tau_{K}$ and $\tau_{KR}$ on the space of Radon measures, corresponding to the classical
Kantorovich and Kantorovich--Rubinshtein (or Fortet--Mourier) norms. We show that the topology $\tau_{KR}$ coincides with
the weak topology on the cone of nonnegative measures and also on bounded uniformly tight sets of measures. For a separable space,
it coincides with the weak topology on weakly compact sets. A~simple  sufficient compactness
condition is obtained for~$\tau_K$,
which combines the uniform tightness with some uniform integrability of quasi-metrics defining the topology.
Kantorovich type seminorms can be used for analysis of barycenters. Along these lines we show convergence of the
barycenters of weakly convergent logarithmically concave and stable measures. Moreover, the same is proved for measures given
by polynomial densities of a fixed degree with respect to logarithmically concave measures.

\section{Kantorovich and Kantorovich--Rubinstein seminorms}

We recall  (see details in \cite{Eng})
that the topology of a completely regular space $X$ is generated by a family of pseudo-metrics  $\Pi$ (a pseudo-metric  differs
from a metric by the property that it can be zero on distinct elements).

A finite nonnegative Borel measure $\mu$ on a topological space is called Radon if for
every Borel set $B$ and every $\varepsilon>0$ there exists a  compact set $K\subset B$ such that
$\mu(B\backslash K)< \varepsilon$. A signed Borel measure $\mu$ is called Radon if
such is its total variation~$|\mu|=\mu^{+}+\mu^{-}$, where $\mu^{+}$ and $\mu^{-}$ are the positive and negative parts of the measure~$\mu$.
On Radon measures, see \cite{B07}.

A set $M$ of Radon measures on $X$ is called uniformly tight if
for every $\varepsilon>0$ there is a compact set $K$ such that $|\mu|(X\backslash K)<\varepsilon$
for all $\mu\in M$.

Let $\mathcal{M}_r(X)$ denote the linear space of all Radon measures on~$X$ and let
$\mathcal{M}_r^{+}(X)$ and $\mathcal{P}_r(X)$ be its subsets consisting of nonnegative measures and probability measures, respectively.

Let $\mathcal{M}_r^{\Pi}(X)$ denote the subset of $\mathcal{M}_r(X)$ consisting of measures $\mu$ for which the function
$x\mapsto p(x,x_0)$ belongs to $L^1(|\mu|)$ for all $p\in \Pi$ for some (then for all) $x_0\in X$.
Note that this class depends on our choice of $\Pi$; say, for $X=(0,1)$ with the usual metric and the family $\Pi$ reducing to it,
all measures belong to $\mathcal{M}_r^{\Pi}(X)$, but the situation changes if for $\Pi$ we take all continuous metrics.
In the case of a normed space $X$ for $\Pi$  we shall take only the norm and in the case of a locally convex space
for $\Pi$ we shall take a collection of seminorms defining the topology (which amounts to taking the collection
of all continuous seminorms); in these cases we shall write
$\mathcal{M}_r^{1}(X)$ in place of $\mathcal{M}_r^{\Pi}(X)$ and speak of measures with finite first moment.

The image of a Radon measure $\mu$ on $X$ under a continuous mapping $F$ to a topological space $Y$ is the Radon
measure $\mu\circ F^{-1}$ given by the equality
$$
\mu\circ F^{-1}(B)=\mu(F^{-1}(B)).
$$

The weak topology on $\mathcal{M}_r(X)$ is the topology of duality with the space $C_b(X)$ of bounded continuous
functions, which is generated by all  seminorms of the form
$$
\mu\mapsto \biggl|\int_X f\, d\mu\biggr|, \quad f\in C_b(X).
$$

The space $\mathcal{M}_r(X)$ of all Radon measures on a metric space $(X,d)$ can be equipped with
the classical Kantorovich--Rubinstein norm
$$
\|\mu\|_{KR,d}=\sup \biggl\{ \int f\, d\mu\colon f\in {\rm Lip}_1(d), \ |f|\le 1\biggr\},
$$
where ${\rm Lip}_1(d)$ is the class of $1$-Lipschitz functions $f$, i.e.,
$|f(x)-f(y)|\le d(x,y)$.
The subspace $\mathcal{M}_r^1(X)$ of all
 measures for which for some $x_0\in X$
(then for all $x_0$) the function $d(x,x_0)$ is integrable can be equipped with the Kantorovich norm
$$
\|\mu\|_{K,d}=\sup\biggl\{ \int f\, d\mu\colon f\in {\rm Lip}_1(d),\, f(x_0)=0\biggr\}+|\mu(X)|.
$$

The topology generated by the Kantorovich--Rubinstein norm coincides with the weak topology on the cone of nonnegative
measures and also on compact sets in the weak topology, although on the whole space these two topologies are uncomparable in nontrivial cases.

For a general completely regular space $X$ both norms have natural analogs
in the form of collections of seminorms: for every pseudo-metric $d$ one can  define $\|\mu\|_{KR,d}$ and $\|\mu\|_{K,d}$
on the  corresponding spaces. It is shown below that the topology $\tau_{KR}$ generated by all such Kantorovich--Rubinstein type
seminorms  coincides with the weak topology on the cone of nonnegative
measures and also on weakly  compact sets (for a separable space). We give a simple sufficient condition for the compactness
in the topology $\tau_{KR}$ and  the similarly defined topology $\tau_{K}$. In the case of a locally convex space $X$ we deduce from this
convergence of the barycenters of weakly convergent measures (a precise formulation is give below). It is also shown that if $X$
is a Fr\'echet space, then every set  compact in the topology $\tau_{K}$ is concentrated on some
separable reflexive Banach space~$E$  compactly embedded into $X$ and is compact in the
topology $\tau_{K}$ on $E$ generated by the stronger topology from~$E$.

Let us connect the topology $\tau_{KR}$ with the weak topology. The first assertion in the next theorem can be derived
from \cite[Lemma~1.3.3]{Cast}, but we include a short justification for completeness.

\begin{theorem}\label{KR}
Suppose that the topology in $X$ is generated by a family of pseudometrics~$\Pi$.
Then the weak topology on the set $\mathcal{M}^+_r(X)$ is
generated by the family of seminorms $\|\cdot\|_{KR, p}$, $p\in \Pi$.

In addition, the weak topology is generated by these seminorms on every  bounded in variation and uniformly tight set
in $\mathcal{M}_r(X)$.
\end{theorem}
\begin{proof}
Let $p\in \Pi$. We denote by $X/p$ the quotient space with the metric
$$
\widehat{p}([x]_p,[y]_p)=p(x,y),\quad [x]_p=\{y\in X\colon p(x,y)=0\}.
$$
The canonical mapping
$
\pi_p\colon x\mapsto [x]_p
$
is continuous. Note that the equality
$$
\|\mu\circ \pi_p^{-1}\|_{\widehat{p}}=\|\mu\|_{KR,p}
$$
is true.
Indeed,
for every function $f\in\mathrm{Lip}_1(\widehat{p})$ on $X/p$ the composition $f\circ\pi_p$ belongs to the set $\mathrm{Lip}_1(p)$,
so $\|\mu\circ \pi_p^{-1}\|_{\widehat{p}}\le\|\mu\|_{KR,p}$.
On the other hand, if $f\in\mathrm{Lip}_1(p)$,
then we set $g([x]_p):=f(x)$. The function $g$ is well-defined, since the function $f$ is Lipschitz in the pseudo-metric $p$.
Moreover, $g\in\mathrm{Lip}_1(\widehat{p})$ and we have
$$
\int_X f\,d\mu=\int_{X/p}g\,d(\mu\circ\pi_p^{-1}).
$$
It suffices to prove
the coincidence of the weak topology and $\tau_{KR}$ on $\mathcal{P}_r(X)$. If a net of measures $\mu_\alpha\in\mathcal{P}_r(X)$
 converges weakly to a measure $\mu\in\mathcal{P}_r(X)$,
then their images $\mu_\alpha$ under the mapping $\pi_p$ are Radon on the metric space $X/p$ and  converge weakly
to the image of~$\mu$. Therefore,
$$
\|\mu_\alpha\circ\pi_p^{-1}-\mu_\circ\pi_p^{-1}\|_{\widehat{p}}=\|\mu_\alpha-\mu\|_{KR,p}\to 0.
$$
Conversely, let $\mu_\alpha\to\mu$ with respect to all seminorms $\|\cdot\|_{KR,p}$. Then we have
convergence of the integrals of all functions of the form $f(x)=\min(p(x,x_0),c)$. This implies convergence
$\mu_\alpha(U)\to\mu(U)$ on all open sets $U$ of the form $U=\{x\colon f(x)<t\}$ with $\mu$-zero boundary,
which implies weak convergence (see \cite[Theorem~4.3.11]{B18}).

Let us prove the second assertion.
Let $S$ be a bounded and uniformly tight set in $\mathcal{M}_r(X)$. Suppose that a net of measures
 $\mu_\alpha$ from $S$  converges weakly to a measure $\mu\in \mathcal{M}_r(X)$. Let us show that for every pseudometric
$p\in \Pi$ we have convergence $\|\mu_\alpha-\mu\|_{KR,p}\to 0$. We can assume that $\|\mu_\alpha\|\le 1$
for all $\alpha$. Given $\varepsilon>0$, we can find a compact set $K$ such that $|\mu|(X\backslash K)<\varepsilon$
and $|\mu_\alpha|(X\backslash K)<\varepsilon$ for all $\alpha$. Note that the set of restrictions to $K$ of the functions
from ${\rm Lip}_1(p)$ bounded by $1$ in absolute value
 is compact in the space $C(K)$ with the sup-norm by the Arzela--Ascoli theorem. Therefore,
it contains a finite $\varepsilon$-net $f_1,\ldots,f_m$. Take an index $\alpha_0$ such that for all $\alpha\ge \alpha_0$ we have
$$
\biggl| \int_X f_i\, d\mu_\alpha - \int_X f_i\, d\mu\biggr|<\varepsilon, \quad i=1,\ldots,m.
$$
Let $f\in {\rm Lip}_1(p)$ and $|f|\le 1$. There is $f_i$ with $|f-f_i|\le \varepsilon$ on $K$. Then
$$
\biggl| \int_X f\, d\mu_\alpha - \int_X f\, d\mu\biggr|< 7\varepsilon,
$$
since the integrals over $K$ differ by at most $\varepsilon$ and the absolute values of the integrals over the complement of $K$ are
estimated by~$\varepsilon$.
\end{proof}

\begin{remark}
\rm
The topologies $\tau_{KR}$ and $\tau_K$ are introduced precisely  in the same way on the space of all Baire measures $\mathcal{M}_\sigma(X)$
or on its subspace  $\mathcal{M}_\tau(X)$ of $\tau$-additive measures (see \cite{B07}).
The previous theorem with the same proof remains valid for $\tau$-additive measures.
\end{remark}

\begin{proposition}
If a completely regular space $X$ is separable or possesses a countable collection of continuous functions separating points,
 then the weak topology coincides with the topology $\tau_{KR}$ on weakly compact sets in $\mathcal{M}_r(X)$.
\end{proposition}
\begin{proof}
Since on a compact space every weaker topology coincides with the original one, it suffices to verify that
 weak convergence of a net of measures from a weakly compact set $S$ implies convergence in the topology
$\tau_{KR}$ under one of our two conditions. Let $X$ be separable and $p\in \Pi$.
Then the image $X_p$ of $X$ under the  indicated factorization is a separable metric space with the completion~$Z_p$.
The image of $S$ is compact in $\mathcal{M}_r(Z_p)$. Since $Z_p$ possesses a countable collection of continuous functions separating points,
the compact image of $S$ is metrizable in the weak topology. Hence it suffices to use that
any weakly convergent sequence in $\mathcal{M}_r(Z_p)$ also converges in the Kantorovich--Rubinshtein norm
(see, e.g., \cite{Pachl79} or \cite[Exercise~3.5.22]{B18}).
The second case is similar: here the compact set $S$ itself is metrizable (because
a countable family of continuous functions separating points gives a
countable family of bounded  continuous functions separating measures), hence its image is also.
Therefore,   it suffices to verify our assertion for countable sequences of Radon
 measures on a metric space, which reduces to the  case of a separable space.
\end{proof}

Note that a similar assertion is true for the space  $\mathcal{M}_\sigma(X)$
of Baire measures on a separable space~$X$. Recall (see \cite[Theorem~4.8.3]{B18}) that a set $M$ in
 $\mathcal{M}_\sigma(X)$ is contained in a weakly compact set precisely when for every sequence
 of  functions $f_n\in C_b(X)$ pointwise decreasing to zero
 (in the formulation of the cited theorem it is mistakenly said ``converging'' in place of ``decreasing'',
 but the proof deals with decreasing sequences; the case of nonnegative measures is covered by Theorem~4.5.10
 with a correct formulation)
 one has
 $$
 \lim\limits_{n\to\infty} \sup_{\mu\in M} \biggl|\int_X f_n\, d\mu\biggr|=0.
 $$
 This criterion extends at once to weakly complete sets in the space of Radon measures
 (or in the case where all Baire measures on $X$ have Radon extensions).

There is a sufficient condition of convergence in the topology $\tau_K$.

\begin{proposition}\label{p2.4}
Suppose that a net $\{\mu_\alpha\}\subset \mathcal{M}_r^{\Pi}(X)$ converges in the topology
$\tau_{KR}$ to a
measure $\mu\in\mathcal{M}_r^{\Pi}(X)$ {\rm(}for nonnegative measures or from measures
in a bounded and uniformly tight family this is equivalent to weak convergence{\rm)}.
If every pseudo-metric $p$ from $\Pi$ satisfies the condition of
uniform integrability
$$
\lim\limits_{R\to\infty} \sup_{\alpha} \int_{\{p\ge R\}} p(x,x_0)\, |\mu_\alpha|(dx) =0 ,
$$
then $\{\mu_\alpha\}$ converges in the topology $\tau_K$. In the case of probability measures
this is also a necessary condition.

Finally, in the case of a countable sequence of measures, convergence in $\tau_{KR}$ can be replaced
with weak convergence.
\end{proposition}
\begin{proof}
Let $p\in \Pi$ and $\varepsilon>0$. There is $R>0$ such that the integral of $p(x,x_0)$ over the set
$\{p\ge R\}$ is less than $\varepsilon$ for every measure~$|\mu_\alpha|$. Next we take an index $\alpha_0$ such that
$\|\mu-\mu_\alpha\|_{KR,p}< \varepsilon (1+R)^{-1}$ for all $\alpha\ge\alpha_0$
and $|\mu(X)-\mu_\alpha(X)|<\varepsilon$. Let $f\in {\rm Lip}_1(p)$ and $f(x_0)=0$.
Set $f_R=\max(-R,\min(f,R))$. Then $f_R=f$ on $\{p<R\}$, $f_R\in {\rm Lip}_1(p)$ and $|f_R|\le R$.
Hence the integrals of $f_R$ against $\mu$ and $\mu_\alpha$ with $\alpha\ge \alpha_0$ differ in absolute value
by at most~$\varepsilon$. Clearly, $|f|\le p$ and $|f_R|\le p$, so the integrals of $f$ and $f_R$ against $\mu_\alpha$
differ in absolute value by at most~$2\varepsilon$. Then the same is true for~$\mu$. Therefore,
the difference of the integrals of $f$ against $\mu$ and $\mu_\alpha$ with $\alpha\ge \alpha_0$ does not exceed
$3\varepsilon$. Hence $\|\mu-\mu_\alpha\|_{K,p}\le 4\varepsilon$.
For nonnegative measures or measures from a bounded uniformly tight family weak convergence is equivalent to convergence
in the topology~$\tau_{KR}$.

It is readily seen that for probability measures
the converse is also true.
Finally, for a countable sequence of measures~$\mu_n$, as above, it suffices to consider the case of a
complete metric space, but then we arrive at the case of a uniformly tight family.
\end{proof}

As one can see from Example \ref{examp} below, for nets of signed measures, convergence
in the topology $\tau_{KR}$ cannot be replaced with weak convergence.

Let us give a sufficient condition for the compactness of sets in $\mathcal{M}_r(X)$ in the topology $\tau_{KR}$
and for sets in $\mathcal{M}_r^{\Pi}(X)$ in the topology $\tau_K$.

\begin{proposition}
Let $S\subset \mathcal{M}_r(X)$ be a
bounded and uniformly tight set. Then $S$ has compact closure in the topology $\tau_{KR}$.

Moreover, if $S\subset \mathcal{M}_r^{\Pi}(X)$ and every pseudo-metric $p$ from $\Pi$ satisfies the condition of
uniform integrability
$$
\lim\limits_{R\to\infty} \sup_{\mu\in S} \int_{\{p\ge R\}} p(x,x_0)\, |\mu|(dx) =0
$$
for some $x_0\in X$, then $S$ is contained in a compact set in the topology $\tau_{K}$.
\end{proposition}
\begin{proof}
It follows from the assumption that $S$ has compact closure in the weak topology, and the previous theorem states that on $S$
it coincides with the topology $\tau_{KR}$. The second  assertion follows by the previous proposition.
Indeed, every net in $S$ contains a subnet $\{\mu_\alpha\}\subset S$ converging weakly and in the topology $\tau_{KR}$.
The limiting measure $\mu$ belongs to $\mathcal{M}_r^{\Pi}(X)$. Indeed, for every $p\in \Pi$ and $R>0$,
letting $f_R(x)=\min (p(x,x_0),R)$ for a fixed point $x_0$, we have
weak convergence of the measures $f_R\cdot \mu_\alpha$ to $f_R\cdot \mu$, which yields the bound
$$
\|f_R\cdot \mu\|\le \sup_\alpha \|f_R\cdot \mu\|\le \sup_\alpha \int_X p(x,x_0)\, |\mu_\alpha|(dx).
$$
So the function $p(x,x_0)$ is $|\mu|$-integrable.
\end{proof}

We  now prove that a uniformly tight family of Radon measures on a Banach space with a uniformly integrable norm remains
uniformly tight with some stronger norm, and this norm is also uniformly integrable
(so that this family is contained in some compact set in the Kantorovich norm). More precisely,
this family turns out to be uniformly tight on a compactly embedded separable reflexive Banach space with
a uniformly integrable norm. The result for a single Borel probability measure on a separable Banach space was proved in~\cite{B0},
extending Buldygin's theorem~\cite{Buld}.
 The proof employs the known Grothendieck's construction (see~\cite[\S 2.5]{BS}).
 Let $B$ be a bounded absolutely convex set in a locally convex space $X$. Denote by $E_B$ the linear span of $B$ equipped with the norm
$$
p_B(x)=\inf\{t>0\colon t^{-1}x\in B\},
$$
which is the Minkowski functional of the set B.
If $X$ is sequentially complete, then $E_B$ is a Banach space.

\begin{theorem}\label{E}
Let $X$ be a Fr\'echet space and let $\mathcal{M}$ be a uniformly tight family of Radon measures on $X$
such that all seminorms $p_n$ generating the topology of  $X$ are uniformly integrable with respect to the measures from $\mathcal{M}$, i.e.,
$$
\lim_{m\to\infty}\sup_{\mu\in\mathcal{M}}\int_{\{x\colon p_n(x)> m\}} p_n(x)\,|\mu|(dx)=0.
$$
Then there is a linear subspace $E\subset X$ with the following properties:

{\rm (i)} the space $E$ with some norm $\|\cdot\|_E$ is a separable reflexive Banach space whose
closed unit ball is compact in $X$;

{\rm (ii)} the family $\mathcal{M}$ is concentrated and uniformly tight on $E$ and $\|\cdot\|_E$ is also uniformly integrable.
\end{theorem}
\begin{proof}
We can assume that all measures $\mu\in\mathcal{M}$ are nonnegative. We need the following technical assertion.
Let $p_n\le p_{n+1}$ for all $n$.
Then there is a sequence of continuous seminorms $q_n$, which generates the original topology of $X$,
and there is sequence of positive numbers $\alpha_n$ decreasing to zero such that $q_n\le q_{n+1}$ and
\begin{equation}\label{c1}
\lim_{n\to\infty}\sup_{\mu\in\mathcal{M}}\sum_{k=n}^{\infty}\mu(x\colon q_k(x)>k\alpha_k)=0,
\end{equation}
\begin{equation}\label{c2}
\lim_{n\to\infty}\alpha_n^{-1}\sup_{\mu\in\mathcal{M}}\int_{\{x\colon q_n(x)> n\alpha_n\}} q_n(x)\,\mu(dx)=0.
\end{equation}
Indeed, there are increasing numbers $N_n$ such that $N_{n+1}> 2^n N_n$ and
$$
\sup_{\mu\in\mathcal{M}}\sum_{k=N_n}^{\infty}\mu(x\colon p_n(x)>k) \le \sup_{\mu\in\mathcal{M}}\int_{\{x\colon p_n(x)> N_n\}} p_n(x)\,\mu(dx)<4^{-n}.
$$
Using these numbers,
we set $\alpha_k=1$ and $q_k=p_1$ if $k\le N_2$, $\alpha_k=2^{-1}$ and $q_k=p_2$ if $N_2<k\le N_3$,
 $\alpha_k=2^{-n}$ and $q_k=p_n$ if $N_{n+1}<k\le N_{n+2}$. Then if $N_{n+1}<k\le N_{n+2}$ we have
 $$
 \alpha_k^{-1}\sup_{\mu\in\mathcal{M}}\int_{\{x\colon q_k(x)> k\alpha_k\}} q_k(x)\,\mu(dx)\le 2^{n}\sup_{\mu\in\mathcal{M}}\int_{\{x\colon p_n(x)> N_n\}} p_n(x)\,\mu(dx)<2^{-n}
 $$
and
\begin{multline*}
\sup_{\mu\in\mathcal{M}}\sum_{k>N_{n+1}}^{\infty}
\mu(x\colon q_k(x)>k\alpha_k)=\sup_{\mu\in\mathcal{M}}\sum_{j=1}^\infty\sum_{N_{n+j}< k\le N_{n+j+1}}\mu(x\colon q_k(x)>k\alpha_k)
\\
=\sup_{\mu\in\mathcal{M}}\sum_{j=1}^\infty\sum_{N_{n+j}< k\le N_{n+j+1}}\mu(x\colon p_{n+j-1}(x)>k2^{-(n+j-1)})\\
\le
\sup_{\mu\in\mathcal{M}}\sum_{j=1}^\infty2^{n+j-1}\sum_{k=N_{n+j-1}}^\infty\mu(x\colon p_{n+j-1}(x)>k)<2^{1-n}.
\end{multline*}
For every $n\in\mathbb{N}$ there is a compact set $K_n$ in the set $U_n:=\{x\colon q_n(x)\le n\}$ such that for
all $\mu\in\mathcal{M}$ we have
$$
\mu(\alpha_n U_n\setminus\alpha_n K_n)<2^{-n}.
$$
Then
$$
\mu\Bigl(X\setminus\bigcup_{n=1}^\infty \alpha_n K_n\Bigr)=0\quad\forall\,\mu\in\mathcal{M},
$$
since
$$
\mu(X\setminus\alpha_n K_n)=\mu(\alpha_n U_n\setminus\alpha_n K_n)+\mu(X\setminus\alpha_n U_n)<2^{-n}+\mu(x\colon q_n(x)> n\alpha_n).
$$
Note that the set
$$
K=\bigcup_{n=1}^\infty c_n K_n,\quad c_n:=\alpha_n n^{-1},
$$
is totally bounded. Indeed, given $\varepsilon>0$, take $n_0\in\mathbb{N}$ and $\delta>0$
such that $\alpha_{n_0}<\delta$ and $\{x\colon q_{n_0}(x)\le \delta\}$ lies in
the open ball of radius $\varepsilon$  in the metric of~$X$ centered at zero.
Then, since the sequence $\{\alpha_n\}$ decreases, we obtain that the compact sets $c_n K_n$ are contained in this
 ball for all $n\ge n_0$. The remaining compact sets are also covered
by finitely many balls of radius $\varepsilon$. The
closed absolutely convex hull  $V$ of $K$ is also precompact (see~\cite[Proposition 1.8.2 ]{BS}), so $V$ is compact. Then $(E_V, p_V)$
is a Banach space (see \cite[Proposition 2.5.1]{BS}). The function $p_V$ is Borel measurable, since $\{x\colon p_V(x)\le c\}=cV$. Moreover,
$$
\alpha_n K_n\subset n K\subset n V,
$$
whence it follows that
$$
\mu(x\colon p_V(x)> n)\le\mu(X\setminus\alpha_n K_n)<2^{-n}+\mu(x\colon q_n(x)>n\alpha_n).
$$
Hence
\begin{align*}
&\int_{\{p_V> n\}} p_V\,d\mu= n\mu(x\colon p_V(x)>n)+\int_{n}^{\infty}\mu(x\colon p_V(x)>t)\,dt\\
&\le n(2^{-n}+\mu(x\colon q_n(x)>n\alpha_n))+\sum_{k=n}^\infty\int_{k}^{k+1}\mu(x\colon p_V(x)>t)\, dt\\
&\le
n2^{-n}+ \alpha_n^{-1}\int_{\{x\colon q_n(x)>n\alpha_n\}} q_n(x)\,\mu(dx)+\sum_{k=n}^\infty(2^{-k}+\mu(x\colon q_k(x)>k\alpha_k)).
\end{align*}
The right-hand side of the last inequality tends to zero as $n\to\infty$ uniformly
in $\mu\in\mathcal{M}$, which implies the uniform integrability of the function $p_V$ with respect to the family of measures $\mathcal{M}$.

We now prove the existence of a separable reflexive Banach space $E$ satisfying conditions (i) and (ii).
There is a convex balanced compact set $W$ such that $V\subset W$, the Banach space $(E_W, p_W)$
is separable and reflexive and $K$ is also compact  in the norm $p_W$ (see \cite[Corollary 2.5.12]{BS}).
Then $p_W\le p_V$,  which shows that $p_W$ is also uniformly integrable. Moreover, all Borel sets in $E_W$ are Borel in $X$,
since the image of a Borel set under a continuous injective mapping from a Polish to a metric space is Borel
(see \cite[Theorem 6.8.6]{B07}). Therefore, the measures from $\mathcal{M}$ can be restricted to the Borel $\sigma$-algebra
of the Banach space $E_W$ and they are concentrated on this space and uniformly tight, since $V$ is compact in $E_W$ and
$$
\mu(E_W\setminus nV)\le\mu(X\setminus\alpha_n K_n)<2^{-n}+\mu(x\colon q_n(x)> n\alpha_n),
$$
which tends to zero as $n\to\infty$ for each measure $\mu$ in $\mathcal{M}$.
\end{proof}

\begin{remark}\label{rem2.6}
\rm
Let $(X,d)$ be a metric space, $x_0\in X$ a fixed point, $q\ge 1$ and $\mathcal{M}^q_r(X)$ the space
of Radon measures with finite moment of order~$q$, that is, measures $\mu$ such that the function
$d(x,x_0)^q$ is $|\mu|$-integrable for some~$x_0$.
Recall that for any $q\ge 1$ the subspace of probability measures  $\mathcal{P}^q_r(X)$
can be equipped with the $q$-Kantorovich metric $d_{K,d,q}$ defined by
$$
d_{K,d,q}^q(\mu,\nu)=\inf_{\sigma\in \Pi(\mu,\nu)} \int d(x,y)^q\, \sigma(dx dy),
$$
where $\Pi(\mu,\nu)$ is the set of probability measures on $X\times X$ with projections $\mu$ and $\nu$ on the factors.
The metric $d_{K,d,q}$ with $q>1$ is not generated by a norm (unlike the case $q=1$, where
$d_{K,d,1}(\mu,\nu)=\|\mu-\nu\|_K$), but the norm
$$
K_{d,q}(\mu)=\|(1+d(\,\cdot , x_0)^q)\mu\|_{KR}
$$
generates the same topology on $\mathcal{P}^q_r(X)$ as $d_{K,d,q}$ (see \cite[Corollary~3.3.7]{B18},
where there is a misprint in the formula: the norm $\|\,\cdot\,\|_K$ should be replaced
with~$\|\,\cdot\,\|_{KR}$).
So in the general case of a family of pseudo-metrics $\Pi$ we can introduced the Kantorovich topology $\tau_{K,q}$
 on $\mathcal{M}^q_r(X)$ generated by the seminorms $K_{p,q}$ with $p\in \Pi$.
\end{remark}

The same reasoning as above leads to the following result for~$\tau_{K,q}$.

\begin{proposition}
Suppose that a net of measures $\mu_\alpha\in \mathcal{M}^q_r(X)$, where $q\ge 1$, converges
to a measure $\mu\in \mathcal{M}_r(X)$ in the topology~$\tau_{KR}$
{\rm(}for nonnegative measures or measures from a bounded and uniformly tight family
this is equivalent to weak convergence{\rm)}. If for every pseudo-metric $p$ from $\Pi$ we have
$$
\lim\limits_{R\to\infty} \sup_{\alpha} \int_{\{p\ge R\}} p(x,x_0)^q\, |\mu_\alpha|(dx) =0 ,
$$
then $\mu\in \mathcal{M}^q_r(X)$ and $\{\mu_\alpha\}$ converges to $\mu$ in the topology $\tau_{K,q}$.
\end{proposition}

\section{Convergence of barycenters}

We say that a Borel measure $\mu$ on a locally convex space $X$ has a mean (or barycenter) $m_\mu\in X$
if $X^*\subset L^1(|\mu|)$ and for every $f\in X^*$ we have
$$
f(m_\mu)=\int_X f(x)\,\mu(dx).
$$
In the case of a Banach space $X$ with a Radon measure $\mu$, the mean exists
if the norm is $\mu$-integrable. In this case $m_\mu$ is the Bochner integral
$$
m_\mu=\int_X x\,\mu(dx).
$$
A similar statement is true in any quasi-complete locally
convex space (see \cite[Corollary~5.6.8]{BS}): for the existence of the mean, it is sufficient
to have the integrability of all seminorms from a family  generating the topology of this space
(which is equivalent to the integrability of all continuous seminorms).

It is worth noting (although we do not use it below) that consideration of convergence
of barycenters in locally convex spaces reduces to Banach spaces by means of the factorizations used above and the
following simple observation: if a net of elements $v_\alpha$
in a locally convex space $X$ and an element $v\in X$ are such that $Tv_\alpha \to Tv$ for every
continuous linear operator $T$ on $X$ with values in a normed space, then
$v_\alpha\to v$ in $X$. Indeed, for each continuous seminorm $p$ on $X$
the linear subspace $Y= p^{-1}(0)$ is closed, so the quotient space $X/Y$ is normed with the norm
$$
\|[x]\|=p(x),\quad [x]=x+Y, \, x\in X,
$$
and the natural projection $x\mapsto [x]$ is linear and continuous.
It follows from this observation that if
a net of measures $\mu_\alpha\in \mathcal{M}_r(X)$
and a measure $\mu\in \mathcal{M}_r(X)$ with barycenters
in a locally convex space $X$
are such that for each normed space $Y$ and each continuous linear  operator
$T\colon X\to Y$ the barycenters of $\mu_\alpha\circ T^{-1}$ converge to the barycenter of $\mu\circ T^{-1}$,
then $m_{\mu_\alpha}\to m_\mu$. Indeed,
the barycenter of $\mu_\alpha\circ T^{-1}$ is~$Tm_{\mu_\alpha}$.

Note also that if a Borel measure $\mu$ has a barycenter and a continuous seminorm $q$ is $\mu$-integrable,
then from the Hahn--Banach theorem and the definition of the barycenter we obtain
$$
 q(m_\mu)=\sup\biggl\{\biggl|\int_X f\,d\mu\biggr|\colon f\in X^*,\,|f|\le q\biggr\}\le\|\mu\|_{K,q}.
 $$

As a consequence of the results of the previous section
(see Proposition~\ref{p2.4}), we obtain the following sufficient condition for convergence of barycenters.

\begin{corollary}\label{means}
Suppose that a sequence of Radon measures $\mu_n$ and a Radon measure $\mu$ on a
locally convex space $X$ have barycenters, the measures $\mu_n$ converge
weakly to $\mu$ and every continuous seminorm is uniformly integrable with respect to the sequence~$\mu_n$. Then $m_{\mu_n}\to m_\mu$.

In the case of probability measures, the same is true for nets.
\end{corollary}

The following example shows that the second assertion of Corollary \ref{means} can fail for a net of signed measures.

\begin{example}\label{examp}
\rm
In the Banach space $X=l^1$ there is a net of uniformly bounded signed measures concentrated on the unit ball such
that it converges weakly to zero,
but their barycenters have unit norms.
We fix a finite collection of bounded continuous functions $f_1,\ldots,f_n$ on~$X$.
Consider  the following vectors in $\mathbb{R}^n$:
$$
v_j=(f_1(e_j),\ldots,f_n(e_j)),\quad j=1,\ldots,n+1,
$$
where $\{e_j\}$ is the standard basis in $l^1$. The vectors $v_j$
are linearly dependent, so there are numbers $c_1,\ldots,c_{n+1}$ not vanishing simultaneously such that
$$
\sum_{j=1}^{n+1}c_j v_j = 0,
$$
or, equivalently,
$$
\sum_{j=1}^{n+1}c_jf_i(e_j) = 0,\quad i=1,\ldots,n.
$$
By normalization, we can achieve the equality $\sum_j |c_j|=1$.
Now for the basic neighborhood of zero in the weak topology
$$
U=U_{f_1,\ldots,f_n, \varepsilon}=\biggl\{\mu\in\mathcal{M}\colon \biggl|\int_X f_i\, d\mu\biggr|<\varepsilon,\,i=1,\ldots,n\biggr\}
$$
we define a measure by the formula
$$
\mu_{U}:=\sum_{j=1}^{n+1} c_j\delta_{e_j}.
$$
Then by construction $\mu_{U}\in U$, and this measure is concentrated on the unit sphere.
The set of basic neighborhoods of zero is directed with respect to the inverse inclusion:
a neighborhood $V$ is declared to be larger than a neighborhood $U$ if $V\subset U$.
 By definition, the constructed net of measures $\mu_U$ converges weakly to zero.
Finally, the mean of the measure $\mu_U$ equals $\sum_j c_j e_j$, therefore, we have
the equality
$\|m_{\mu_U}\|=\sum_{j=1}^{n+1}|c_j|=1$.
\end{example}

In this example all measures in the net are absolutely continuous  with respect to the measure $\sum_n 2^{-n}\delta_{e_n}$.
Note that similarly one can construct a net converging in the stronger
topology of duality with the space of all bounded Borel functions.

In the general case,  weak convergence of measures $\mu_n$ to $\mu$ and
weak convergence of measures $\nu_n=f_n\cdot\mu_n$ to a measure $\nu$ do not imply that
$\nu$ is absolutely continuous with respect to $\mu$. For example, the measures $(1-n^{-1})\delta_0+ n^{-1}\delta_1$ converge
weakly (and even in the total variation norm)
to $\delta_0$, but the measure $\delta_1$ is mutually singular with~$\delta_0$. However, under the
following additional condition this implication is true.

\begin{lemma}\label{lem2}
Suppose that  Radon probability measures $\mu_\alpha$ on a completely
regular space $X$  converge weakly to a Radon measure $\mu$ and the measures $\nu_\alpha=f_\alpha\cdot\mu_\alpha$
converge weakly to a Radon measure $\nu$. Assume also that
$$
\lim\limits_{R\to\infty} \sup_\alpha |\nu_\alpha|(x\colon |f_\alpha(x)|\ge R)=0.
$$
Then $\nu\ll \mu$. In particular, this is true if $\sup_\alpha \|f_\alpha\|_{L^2(\mu_\alpha)}<\infty$.
\end{lemma}
\begin{proof}
Let $K$ be a   compact set such that $\mu(K)=0$. Suppose that $\nu(K)=\delta>0$ (the case $\nu(K)<0$ is similar).
Pick  $R>1$ such that
$$
\sup_\alpha |\nu_\alpha|(x\colon |f_\alpha(x)|\ge R)< \delta/4.
$$
We can find an   open set $U$ such that $K\subset U$ and $\mu(\overline{U})=\mu(U)<\delta (2R)^{-1}$,
where $\overline{U}$ is the closure of $U$.
This is possible, since we can take some open set $U_0$ with $K\subset U_0$ and
$\mu(U_0)<\delta (2R)^{-1}$, then find a continuous function $f$ with values in $[0,1]$
for which $f|_K=1$ and $f|_{X\backslash U_0}=0$, finally, for $U$ we can take the set $\{f>c\}$,
where $c\in (0,1)$ is picked such that $\mu(f^{-1}(c))=0$.
Then $\mu_\alpha(U)\to\mu(U)$ by Alexandrov's criterion (see \cite[Corollary~4.3.5]{B18}),
hence $\mu_\alpha(U)<\delta (2R)^{-1}$ for all $\alpha$ large enough.
For such $\alpha$ we finally obtain
$$
|\nu_\alpha|(U)=\int_U |f_\alpha|\, d\mu_\alpha \le R \mu_\alpha(U)+ \delta/4< 3\delta/4,
$$
which gives the estimate $|\nu(U)|\le \delta$, hence $\nu(K)\le \delta$, which contradicts our assumption.
\end{proof}

\section{Logarithmically concave and stable measures}

Now we investigate  convergence of  logarithmically concave measures and their means. Recall that
a Radon probability  measure $\mu$
on a locally convex space $X$ is called logarithmically concave if $\mu$ satisfies the inequality
$$
\mu(tA+(1-t)B)\ge\mu(A)^t\mu(B)^{1-t}
$$
for all compact sets $A$ and $B$.
This definition is also equivalent to the property
 that for every continuous linear operator $T$ from $X$ to $\mathbb{R}^n$  the measure
$\mu\circ T^{-1}$ has a density of the form $\exp (-V)$ with respect
to Lebesgue measure on some affine subspace with a convex function $V$ (see~\cite{Bor}, \cite{B10}).

The class of logarithmically concave measures contains all Gaussian measures, i.e., measures for which all continuous linear functionals
are Gaussian random variables.

We need the following estimate due to C. Borell
(see \cite{Bor} or \cite[Theorem~4.3.7]{B10}).

Let $\mu$ be a logarithmically concave measure on a locally convex space $X$ and let $A$ be an absolutely
 convex Borel set with $\theta:=\mu(A)>0$. Then
\begin{equation}
\label{conc}
\mu(X\setminus{tA})\le\Bigl(\frac{1-\theta}{\theta}\Bigr)^{t/2},\quad t\ge 1.
\end{equation}
This estimate implies that, for any Borel seminorm $q$ such that $\mu(q>1)=1-\theta<1/2$, one has
\begin{equation}
\label{expq}
\int_X \exp(\alpha (\theta)q)\, d\mu\le M(\theta)
\end{equation}
with some constants $\alpha(\theta)$ and $M(\theta)$ depending only on  $\theta>1/2$.
Therefore,
\begin{equation}
\label{Lp}
\|q\|_{L^p(\mu)}\le C(p,\theta)
\end{equation}
with some constants $C(p,\theta)$ depending only on $p$ and $\theta>1/2$.

We apply inequality \eqref{conc} to prove the following sufficient condition for convergence of
means of logarithmically concave measures.
Note that all Radon Gaussian measures have barycenters, but for logarithmically concave measures this
is known only under the assumption of
sequential completeness of the space, as for general measures with finite first moment.

\begin{theorem}
\label{mean conc}
If a net of logarithmically concave measures $\mu_\alpha$ converges weakly  to
a measure $\mu$,
then for every continuous seminorm $q$ there is $\kappa>0$ such that
$$
\lim_{\alpha}\int_X \exp(\kappa q)\,d\mu_\alpha=\int_X\exp(\kappa q)\,d\mu.
$$
Therefore, for every $r>0$ we have
$$
\lim\limits_{\alpha}\int_X q^r\,d\mu_\alpha=\int_X q^r\,d\mu.
$$
Finally, if $\mu_\alpha, \mu$ have barycenters, then $m_{\mu_\alpha}\to m_\mu$.
\end{theorem}
\begin{proof}
Take $c>0$ such that $\mu(q<c)>\theta>1/2$.
Since $\{\mu_\alpha\}$ converges weakly to $\mu$ and the set $\{q<c\}$ is open, by Alexandrov's theorem we have $\mu_\alpha(q<c)>\theta$
for all $\alpha$ larger than some $\alpha_0$. Then by (\ref{conc}) we obtain the inequality
$$
\mu_\alpha(q\ge ct)\le \Bigl(\frac{1-\theta}{\theta}\Bigr)^{t/2},\quad t\ge 1.
$$
Set
 $$
 \tau:=\Bigl(\frac{1-\theta}{\theta}\Bigr)^{1/2},\quad \tau\in (0,1).
 $$
For all $\alpha\ge \alpha_0$ we have
\begin{align*}
\int_X \exp(\kappa q)\,d\mu_\alpha
&= 1 + c\kappa\int_{0}^{\infty}\exp(c\kappa t)\mu_\alpha(q\ge ct)\,dt\\
&\le \exp(c\kappa)+c\kappa\int_{1}^{\infty}\exp(c\kappa t)\tau^t\, dt.
\end{align*}
The integral
$$
\int_{1}^{\infty}\exp(c\kappa t)\tau^t\, dt
$$
is finite for $\kappa <-\ln\tau/c$. Thus, there exists $\kappa_0>0$ such that for all $\alpha\ge \alpha_0$ the inequality
$$
\int_X \exp(\kappa_0 q)\,d\mu_\alpha\le I(\theta,c)
$$
holds, whence for all $\kappa<\kappa_0$ we have the estimate
$$
\int_{\{q\ge R\}} \exp(\kappa q)\,d\mu_\alpha\le \exp((\kappa-\kappa_0)R)I(\theta,c).
$$
Consequently, the integrals of $\exp(\kappa q)$
converge (see \cite[Theorem 4.3.15]{B18}), which yields convergence of the integrals of $q^r$.
Convergence of means follows from Corollary \ref{means}.
\end{proof}

\begin{corollary}
In the previous theorem one has also convergence in the topology $\tau_{K,q}$ with any $q\ge 1$ introduced
in Remark~{\rm\ref{rem2.6}}.
\end{corollary}

For a Radon probability measure $\mu$ on a locally convex space, we denote by
$\mathcal{P}^d(\mu)$ the set of all $\mu$-measurable polynomials of degree $d\ge 0$, i.e.,
$\mu$-measurable functions possessing versions that are  polynomials of degree  $d$ on $X$  in the usual
algebraic sense (this is equivalent to the property that the restrictions to all affine lines are polynomials of degree~$d$).
For a Gaussian measure, every measurable polynomial of degree  $d$ is the limit almost everywhere and in $L^2$ of a sequence
of polynomials of the form $f(l_1,\ldots,l_n)$, where $f$ is a polynomial of degree  $d$ on $\mathbb{R}^n$ and $l_j$ are continuous
linear functionals. It is not known whether this is true for all logarithmically  concave measures; about measurable
polynomials, see~\cite{B16}.

For measurable polynomials on a  space equipped with a  logarithmically  concave measure two very important
estimates are known with constants independent of the measure. The first one
 (obtained in \cite{Carb}, \cite{Naz} in the finite-dimensional case and extended in \cite{ArKos}
 to the infinite-dimensional case) gives an estimate for small values:
\begin{equation}\label{sm-val}
\mu(x\colon |f(x)|\le r)\|f\|_{L^1(\mu)}^{1/d} \le cd r^{1/d}, \quad f\in \mathcal{P}^d(\mu).
\end{equation}
The second one (see \cite{Bob}, \cite{Bob02}, \cite{ArKos}) gives the equivalence of all $L^p$-norms on $\mathcal{P}^d(\mu)$:
$$
\|f\|_{L^p(\mu)}\le C(p,d)\|f\|_{L^1(\mu)}, \quad f\in \mathcal{P}^d(\mu).
$$

\begin{corollary}
{\rm(i)} Let $\{\mu_\alpha\}$ be a net of logarithmically  concave measures on a sequentially complete locally convex space~$X$ and let
$\nu_\alpha=f_\alpha\cdot\mu_\alpha$ be probability  measures, where $f_\alpha\in \mathcal{P}^d(\mu_\alpha)$
with a common degree~$d$.
If the measures $\nu_\alpha$   converge weakly to a measure $\nu$,
then their means  $m_{\nu_\alpha}$ converge to $m_\nu$.

In addition, if $\{\nu_\alpha\}$ is uniformly tight,
which holds automatically in the  case of a weakly convergent countable sequence on a Fr\'echet space,
then $\{\mu_\alpha\}$ is also uniformly tight and has a limit point~$\mu$ that is a
logarithmically  concave measure, moreover, the measures $\mu$ and $\nu$ are equivalent.

{\rm(ii)} Suppose that $\{\mu_\alpha\}$ is
a uniformly tight family of logarithmically  concave measures  on a locally convex space and for each $\alpha$
there is a measure $\nu_\alpha=f_\alpha\cdot \mu_\alpha$ with $f_\alpha\in \mathcal{P}^d(\mu_\alpha)$. If the family  $\{\nu_\alpha\}$ is bounded
in variation, then it is uniformly tight.
\end{corollary}
\begin{proof}
(i) Let us verify that every continuous seminorm $q$ on $X$ is uniformly integrable with respect to the measures~$|\nu_\alpha|$ and also
with respect to the measures $\mu_\alpha$. It suffices to verify the uniform boundedness of the norms $\|q\|_{L^2(\mu_\alpha)}$,
since the equality $\|f_\alpha\|_{L^1(\mu_\alpha)}=1$
implies the uniform follows boundedness of the norms $\|f_\alpha\|_{L^2(\mu_\alpha)}$.
Let $\varepsilon>0$. By \eqref{sm-val} there holds the estimate
$$
\mu_\alpha(f_\alpha\le r)\le cd r^{1/d}.
$$
Pick $r\in (0,1)$ such that $cd r^{1/d}<\varepsilon/2$.
Next we find  $t>1$ such that $2^{-t/2}<\varepsilon r/2$. There is a number  $R>0$ for which
 $\nu(\{q< R\})>2/3$. By weak convergence $\nu_\alpha(\{q< R\})>2/3$ for all $\alpha$ large enough.
We can assume that this is true for all~$\alpha$. By inequality \eqref{conc} we have
$$
\nu_\alpha(\{q\ge tR\})< \varepsilon r/2.
$$
Then
$$
\mu_\alpha (\{q\ge tR\}\cap \{f_\alpha>r\})\le r^{-1} \nu_\alpha (\{q\ge tR\})< \varepsilon/2,
$$
whence
$$
\mu_\alpha (\{q\ge tR\})\le \varepsilon/2+\mu_\alpha (\{f_\alpha\le r\})< \varepsilon.
$$
If the family $\{\nu_\alpha\}$ is uniformly tight, then the uniform tightness of $\{\mu_\alpha\}$ is obvious from~\eqref{sm-val}.

It follows from Lemma \ref{lem2} that the measure $\nu$  is absolutely continuous with respect to~$\mu$.
The absolute continuity of $\mu$ with respect to $\nu$ follows from  the same lemma applied to the
probability measures $f_\alpha\cdot \mu_\alpha$ and the measures $\mu_\alpha$ given by the
 densities $f_\alpha^{-1}$ with respect to $f_\alpha\cdot \mu_\alpha$
(note that $f_\alpha(x)>0$ for $\mu_\alpha$-a.e.~$x$).
These densities satisfy the hypotheses of the lemma, since
$$
\mu_\alpha(x\colon f_\alpha(x)^{-1}\ge R)=\mu_\alpha(x\colon f_\alpha(x)\le R^{-1})\le cd R^{-1/d}
$$
by estimate \eqref{sm-val}.

(ii) The norms $\|f_\alpha\|_{L^2(\mu_\alpha)}$ are uniformly bounded by some number $M$ as explained above.
Hence for every Borel set $B$ we have $|\nu_\alpha|(B)\le M^{1/2}\mu_\alpha(B)^{1/2}$,
which yields the claim.
\end{proof}

It remains unclear in the considered situation with a countable sequence
whether the measures  $\mu_n$ must converge (even if they are  uniformly
tight, it is not clear whether the limit point is unique).
This is unclear even in the  case of Gaussian measures~$\mu_n$ (if they are different).

Stable measures form another important class of probability distributions (see \cite{B07}, \cite{T}, \cite{Z}).
Recall that a Radon probability  measure $\mu$ on a locally convex space $X$ is called stable of order $p\in (0,2]$ if
for every $\alpha>0$ and $\beta>0$ there is a vector $v$ such that
the image of $\mu$ under the mapping $x\mapsto (\alpha^p+\beta^p)^{1/p}x+v$ equals the convolution
of the images of $\mu$ under the homotheties with the coefficients $\alpha$ and $\beta$.
In other words, if $\xi$ and $\eta$ are independent random vectors with distribution $\mu$, then
$\alpha \xi +\beta \eta$  has the same law as $(\alpha^p+\beta^p)^{1/p}\xi+v$.
The case $p=2$ corresponds to Gaussian measures, and this is the only intersection
with the class of logarithmically concave measures. Stable measures of order $p>1$ possess barycenters.
Indeed, as shown in \cite{Acosta75}, in this case all measurable seminorms are integrable, hence
the barycenter exists in the case of a complete space $X$, so it exists in the completion,
but it is readily seen from the definition that it must belong to the original space.
Note also that if a net of measures $\mu_\alpha$ that are stable of orders greater than some $p_1>1$ converges
weakly to a Radon measure $\mu$, then $\mu$ is also stable of some order $p\ge p_1$.
Indeed, it is known (see \cite{DK}) that if all one-dimensional projections of a measure $\nu$ are stable, then
they are stable of the same order $\gamma$, and if $\gamma>1$, then $\nu$ is stable of order $\gamma$.
In order to apply this result from~\cite{DK} we can take a linear topological embedding of $X$ into a suitable power $\mathbb{R}^T$
of the real line and obtain that $\nu$ is stable on~$\mathbb{R}^T$, but then it remains stable on~$X$, because $X$ is
$\nu$-measurable in~$\mathbb{R}^T$ by the Radon property (there is a sequence of compacts sets $K_n$ in $X$ with $\nu(K_n)\to 1$ and
these sets are also compact in~$\mathbb{R}^T$).
Hence
it suffices to consider the one-dimensional case, where there is a countable sequence $\mu_{\alpha_{n}}$
of elements of the original net converging to $\mu$. Passing to a subsequence we can assume that
the orders $p_n$ of $\mu_{\alpha_{n}}$ converge to some $p\in [p_1,2]$.
The Fourier transform of $\mu_{\alpha_{n}}$ has the form (see \cite{Z})
$$
\exp[ita_n -c_n |t|^{p_n}(1- i b_n {\rm sign}( t) \tan (\pi p_n/2))].
$$
It follows that $a_n\to a$, $b_n\to b$, $c_n\to c$, so that the Fourier
transform of $\mu$ has the same form with $(a,c,p,b)$, hence is stable of order~$p$.

\begin{theorem}
Suppose that a net of measures $\mu_\alpha$ that are stable of orders greater than some $p_1>1$ converges
weakly to a Radon measure $\mu$. Then they converge in the Kantorovich topology.
Hence their barycenters converge to the barycenter of $\mu$.
\end{theorem}
\begin{proof}
As explained above, the measure $\mu$ is also stable of some order in $[p_1,2]$.
It suffices to show that for every continuous seminorm $q$ there is a number $r>1$ such that
the integrals of $q^r$ with respect to $\mu_\alpha$ are uniformly bounded for $\alpha$ larger than some $\alpha_0$.

We fix a number $\delta$ such that $2^{-1/2}<\delta<1$ and then
pick $k>1$ such that $(2^{1/2}\delta)^k>3$.
Next, set
$$
\varrho=\prod_{n=1}^\infty (1+\delta ^n).
$$
It is seen from the proof of \cite[Lemma~3.3]{Acosta75}
that if $\mu$ is a stable measure of order $p\ge p_1$ and $q$ is a measurable seminorm such that
$\mu(q<1)>3/4$, then
$$
\sum_{n=1}^\infty \mu(q>\beta^n)\le 4k , \ \beta=2^{1/p}\delta.
$$
Indeed, in the notation of the cited lemma one has $\beta (\beta^k -2^{1/p})>1$,
because $2^{1/p}\le 2$ and $\beta\ge 2^{1/2}\delta$ due to the bound $p\le 2$. Hence
$$\gamma=\mu(q\le \beta(\beta^k-2^{1/p}))>3/4,$$
which yields the bound $2\gamma/(2\gamma-1)<4$, but in the cited lemma
the series above is estimated by $k2\delta/(2\delta-1)$.

Next, an easy inspection of the proof of \cite[Theorem~3.1]{Acosta75} shows that for a strictly stable measure one has
$$
\mu(q>t)\le C t^{-p}, \ t>0,
$$
where
\begin{align*}
C &\le 2C_n''\varrho^p=2\varrho^p/C_n'=2\varrho^p/\prod_{n=1}^\infty \mu(q\le \beta^n)
\\
&=
2\varrho^p/\prod_{n=1}^\infty (1-\mu(q>\beta^n))
\le
2\varrho^p\exp \Bigl(2\sum_{n=1}^\infty \mu(q>\beta^n)\Bigr)
\le 2\varrho^p\exp (8k).
\end{align*}
Finally, in the general case
$$
\mu(q>t)\le 2C (t-1)^{-p}\le 4\varrho^p\exp (8k)(t-1)^{-p} , \, t>1.
$$
Clearly, in our situation we can assume that $\mu(q<1)>3/4$, so $\mu_\alpha(q<1)>3/4$ for all $\alpha$ sufficiently large.
It follows that we have
$$
\mu_\alpha(q>t)\le 8\varrho^{p_\alpha}\exp (8k)t^{-p_1}, \, t>2.
$$
This estimate completes the proof.
\end{proof}

\begin{corollary}
In the previous theorem one has also convergence in the topology $\tau_{K,q}$ with any $q< p_1$ introduced
in Remark~{\rm\ref{rem2.6}}.
\end{corollary}

Konstantin Afonin: wert8394@gmail.com

Vladimir Bogachev: vibogach@mail.ru

\end{document}